\documentclass[leqno]{article}
\usepackage{geometry}
\usepackage{graphicx}	
\usepackage[cp1251]{inputenc}
\usepackage{mathtools}
\usepackage{amsfonts,amssymb,mathrsfs,amscd,amsmath,amsthm}
\usepackage{verbatim}

\usepackage{url}

\def\ig#1#2#3#4{\begin{figure}[!ht]\begin{center}%
\includegraphics[height=#2\textheight]{#1.eps}\caption{#4}\label{#3}%
\end{center}\end{figure}}

\def\thtext#1{
  \catcode`@=11
  \gdef\@thmcountersep{. #1}
  \catcode`@=12
}

\def\threst{
  \catcode`@=11
  \gdef\@thmcountersep{.}
  \catcode`@=12
}

\theoremstyle{plain}
\newtheorem{thm}{Theorem}[section]
\newtheorem{prop}[thm]{Proposition}
\newtheorem{cor}[thm]{Corollary}
\newtheorem{ass}[thm]{Assertion}
\newtheorem{lem}[thm]{Lemma}

\theoremstyle{definition}
\newtheorem{dfn}[thm]{Definition}
\newtheorem{rk}[thm]{Remark}
\newtheorem{notation}[thm]{Notation}
\newtheorem{constr}[thm]{Construction}

 \catcode`@=11
 \def\.{.\spacefactor\@m}
 \catcode`@=12

\def\N{{\mathbb N}}
\def\R{\mathbb R}

\def\a{\alpha}
\def\b{\beta}
\def\e{\varepsilon}
\def\dl{\delta}
\def\D{\Delta}
\def\g{\gamma}

\def\l{\lambda}
\def\r{\rho}
\def\s{\sigma}
\def\S{\Sigma}

\def\0{\emptyset}
\def\:{\colon}
\def\rom#1{\emph{#1}}
\def\({\rom(}
\def\){\rom)}
\def\sm{\setminus}
\def\ss{\subset}

\def\x{\times}

\def\diam{\operatorname{diam}}
\def\dis{\operatorname{dis}}
\def\MST{\operatorname{MST}}
\def\mst{\operatorname{mst}}
\def\np{\operatorname{np}}
\def\opt{{\operatorname{opt}}}

\def\XST{\operatorname{XST}}
\def\xst{\operatorname{xst}}

\def\cC{{\cal C}}
\def\cD{{\cal D}}
\def\cM{{\cal M}}
\def\cP{{\cal P}}
\def\cR{{\cal R}}

\def\tR{{\tilde R}}

\begin{document}
\title{Geometry of Compact Metric Space in Terms of Gromov--Hausdorff Distances to Regular Simplexes}
\author{Alexander~O.~Ivanov, Alexey~A.~Tuzhilin}
\maketitle

\begin{abstract}
In the present paper we investigate geometric characteristics of compact metric spaces, which can be described in terms of Gromov--Hausdorff distances to simplexes, i.e., to finite metric spaces such that all their nonzero distances are equal to each other. It turns out that these Gromov--Hausdorff distances depend on some geometrical characteristics of finite partitions of the compact metric spaces; some of the characteristics can be considered as a natural analogue of the lengths of edges of minimum spanning trees. As a consequence, we constructed an unexpected example of a continuum family of pairwise non-isometric finite metric spaces with the same distances to all simplexes.
\end{abstract}

\section*{Introduction}
\markright{\thesection.~Introduction}
In the present paper we investigate geometric characteristics of compact metric spaces, which can be described in terms of Gromov--Hausdorff distances to simplexes, i.e., to finite metric spaces such that all their nonzero distances are equal to each other. In~\cite{TuzMST-GH} these distances were used to calculate the length of edges of a minimum spanning tree constructed on a finite metric space. In the present paper we generalize the results from ~\cite{TuzMST-GH} to the case of arbitrary compact metric spaces. It turns out that these Gromov--Hausdorff distances depend on some geometrical characteristics of finite partitions of compact metric spaces; some of the characteristics can be considered as a natural analogue of the lengths of edges of minimum spanning trees. We calculate the Gromov--Hausdorff distances from an arbitrary compact metric space to a simplex of sufficiently small or sufficiently large diameter, see Theorems~\ref{thm:spec}, \ref{thm:small1}, and \ref{thm:small2}. For a finite $n$-point metric space we find the distances to an arbitrary simplex consisting of at least $n-1$ points (Theorems~\ref{thm:dist-n-simplex-same-dim}, \ref{thm:dist-n-simplex-bigger-dim}, and \ref{thm:minus_one}). Nevertheless, the general problem of calculating the distance from an arbitrary metric space to an arbitrary simplex remained unsolved yet. We demonstrate non-triviality of the problem by presenting a few examples in the end of the paper. In particular, we show that the set of all distances from a compact metric space to all simplexes is not a metric invariant, i.e., such collections can coincide for non-isometric finite metric spaces. Moreover, we construct an example
of infinite (continuum) set of pairwise non-isometric finite metric spaces having the same collection of those distances.

In this paper we use the technique of irreducible optimal correspondences~\cite{IvaNikolaevaTuz, IvaIliadisTuz, IvaTuzIrreducible}. We show that to calculate the Gromov--Hausdorff distance from a compact metric space $X$ to an $n$-point simplex, where $n$ is less than or equal to the cardinality of $X$, one can consider only those correspondences which generates partitions of the space $X$ into $n$ nonempty disjoint subsets, see Theorem~\ref{thm:m_less_than_n_corresp}. This result has enabled us to advance essentially in calculations of concrete Gromov--Hausdorff distances.

\section{Preliminaries}
\markright{\thesection.~Preliminaries}
For an arbitrary set $X$ by $\#X$ we denote its \emph{cardinality}.

Let $X$ be an arbitrary metric space. The distance between its points $x$ and $y$ is denoted by $|xy|$. If $A,B\ss X$ are nonempty, then we put $|AB|=\inf\bigl\{|ab|:a\in A,\,b\in B\bigr\}$. If $A=\{a\}$, then we write $|aB|=|Ba|$ instead of $|\{a\}B|=|B\{a\}|$.

For a point $x\in X$ and a real number $r>0$ by $U_r(x)$ we denote the open ball of radius $r$ centered at $x$; for any nonempty $A\ss X$ and $r>0$ we put $U_r(A)=\cup_{a\in A}U_r(a)$.

\subsection{Hausdorff and Gromov--Hausdorff distances}
For nonempty $A,\,B\ss X$ we put
$$
d_H(A,B)=\inf\bigl\{r>0:A\ss U_r(B)\ \&\ B\ss U_r(A)\bigr\}=\max\{\sup_{a\in A}|aB|,\,\sup_{b\in B}|Ab|\}.
$$
This value if called the \emph{Hausdorff distance between $A$ and $B$}. It is well-known~\cite{BurBurIva} that the Hausdorff distance is a metric on the family of all nonempty closed bounded subsets of $X$.

Let $X$ and $Y$ be metric spaces. A triple $(X',Y',Z)$ that consists of a metric space $Z$ and its subsets $X'$ and $Y'$ isometric to $X$ and $Y$, respectively, is called a \emph{realization of the pair $(X,Y)$}. The \emph{Gromov--Hausdorff distance $d_{GH}(X,Y)$ between $X$ and $Y$} is the infimum of real numbers $r$ such that there exists a realization $(X',Y',Z)$ of the pair $(X,Y)$ with $d_H(X',Y')\le r$. It is well-known~\cite{BurBurIva} that the $d_{GH}$ is a metric on the family $\cM$ of isometry classes of compact metric spaces.

For various calculations of the Gromov--Hausdorff distances, the technique of correspondences is useful.

Let $X$ and $Y$ be arbitrary nonempty sets. Recall that a \emph{relation\/} between the sets $X$ and $Y$ is a subset of the Cartesian product $X\x Y$.  By $\cP(X,Y)$ we denote the set of all \textbf{nonempty\/} relations between $X$ and $Y$. Let us look at each relation $\s\in\cP(X,Y)$ as at a multivalued mapping, whose domain may be less than $X$. Then, similarly with the case of mappings, for any $x\in X$ and any $A\ss X$ their images $\s(x)$ and $\s(A)$ are defined, and for any $y\in Y$ and any $B\ss Y$ their preimages $\s^{-1}(y)$ and $\s^{-1}(B)$ are also defined.

A relation $R\in\cP(X,Y)$ is called a \emph{correspondence}, if the restrictions of the canonical projections $\pi_X\:(x,y)\mapsto x$ and $\pi_Y\:(x,y)\mapsto y$ onto $R$ are surjective. By $\cR(X,Y)$ we denote the set of all correspondences between $X$ and $Y$.

Let $X$ and $Y$ be arbitrary metric spaces. The \emph{distortion $\dis\s$ of a relation $\s\in\cP(X,Y)$} is the value
$$
\dis\s=\sup\Bigl\{\bigl||xx'|-|yy'|\bigr|: (x,y),(x',y')\in\s\Bigr\}.
$$

\begin{prop}[\cite{BurBurIva}]
For any metric spaces $X$ and $Y$ we have
$$
d_{GH}(X,Y)=\frac12\inf\bigl\{\dis R:R\in\cR(X,Y)\bigr\}.
$$
\end{prop}

For finite metric spaces $X$ and $Y$ the set $\cR(X,Y)$ is finite as well, therefore there always exists an $R\in\cR(X,Y)$ such that $d_{GH}(X,Y)=\frac12\dis R$. Every such correspondence $R$ is called \emph{optimal}. Notice that the optimal correspondences exist also for any compact metric spaces $X$ and $Y$, see~\cite{IvaIliadisTuz}. The set of all optimal correspondences between $X$ and $Y$ is denoted by $\cR_\opt(X,Y)$. Thus, for compact metric spaces $X$ and $Y$ we have $\cR_\opt(X,Y)\ne\0$.

The inclusion relation generates a partial order on $\cR(X,Y)$: $R_1\le R_2$, iff $R_1\ss R_2$. The relations minimal with respect to this order are called \emph{irreducible}, and the remaining ones are referred as \emph{reducible}.  In~\cite{IvaTuzIrreducible} it is proved that for any compact metric spaces $X$ and $Y$ there always exists an irreducible optimal correspondence $R$.  By $\cR_\opt^0(X,Y)$ we denote the set of all irreducible optimal correspondences between $X$ and $Y$. As it was mentioned above, $\cR_\opt^0(X,Y)\ne\0$.

The next well-known facts can be easily proved by means of the correspondences technique. For any metric space $X$ and any positive real $\l>0$ let $\l X$ stand for the metric space which is obtained from $X$ by multiplication of all the distances by $\l$.

\begin{prop}[\cite{BurBurIva}]\label{prop:GH_simple}
Let $X$ and $Y$ be metric spaces. Then
\begin{enumerate}
\item\label{prop:GH_simple:1} If $X$ is a single-point metric space, then $d_{GH}(X,Y)=\frac12\diam Y$\rom;
\item\label{prop:GH_simple:2} If $\diam X<\infty$, then
$$
d_{GH}(X,Y)\ge\frac12\big|\diam X-\diam Y\big|;
$$
\item\label{prop:GH_simple:3} $d_{GH}(X,Y)\le\frac12\max\{\diam X,\diam Y\}$, in particular, for bounded $X$ and $Y$ it holds $d_{GH}(X,Y)<\infty$\rom;
\item\label{prop:GH_simple:4} For any $X,Y\in\cM$ and any $\l>0$ we have $d_{GH}(\l X,\l Y)=\l d_{GH}(X,Y)$. Moreover, for $\l\ne1$ the unique invariant space is the single-point one. In other words, the multiplication of a metric by $\l>0$ is a homothety of $\cM$ centered at the single-point space.
\end{enumerate}
\end{prop}

\subsection{A few elementary relations}
The next relations will be useful for concrete calculations of Gromov--Hausdorff distances.

\begin{prop}\label{prop:max_abs}
For any nonnegative $a$ and $b$ the following inequality holds\rom:
$$
\max\big\{a,|b-a|\big\}\le\max\{a,b\}.
$$
\end{prop}

\begin{proof}
Indeed, if $b\ge a$, then $|b-a|=b-a\le b$. If $b\le a$, then $|b-a|=a-b\le a$.
\end{proof}

\begin{prop}\label{prop:many_abs_dif}
Let $A\ss\R$ be a nonempty bounded subset, and let $t\in\R$. Then
$$
\sup_{a\in A}|t-a|=\max\{t-\inf A,\sup A-t\}=\Big|t-\frac{\inf A+\sup A}{2}\Big|+\frac{\sup A-\inf A}{2}.
$$
\end{prop}

\begin{proof}
Consider the segment $[\inf A,\sup A]$. If $t$ is placed to the left side of the segment middle point, i.e., $t\le(\inf A+\sup A)/2$, then the value $\sup_{a\in A}|t-a|$ is achieved at the right end of the segment, i.e., it is equal to $\sup A-t$; also, $t-\inf A\le|t-\inf A|\le\sup A-t$. Therefore, for such $t$ the proposition holds. One can similarly consider the case $t\ge(\inf A+\sup A)/2$.
\end{proof}

\begin{prop}\label{prop:many_abs_dif_and_t}
Let $A\ss\R$ be a nonempty bounded subset, $\inf A\ge0$, and let $t\in\R$. Then
$$
\sup_{a\in A}\big\{t,|t-a|\big\}=\max\{t,\sup A-t\}.
$$
\end{prop}

\begin{proof}
Let $A'=\{0\}\cup A$, then $\sup_{a\in A'}|a-t|=\sup_{a\in A}|a-t|$. By Proposition~\ref{prop:many_abs_dif}, we have
$$
\sup_{a\in A'}|a-t|=\max\{t,\sup A-t\}.
$$
\end{proof}

\section{Minimum Spanning Trees}
\markright{\thesection.~Minimum spanning trees}
To calculate the Gromov--Hausdorff distance between finite metric spaces, minimum and maximum spanning trees turn out to be useful. Also, the edges lengths of these trees turn out to be closely related to some geometrical properties of various partitions of the ambient space.

Let $G=(V,E)$ be an arbitrary (simple) graph with the vertex set $V$ and the edge set $E$. If $V$ is a metric space, then  the \emph{length $|e|$ of edge $e=vw$ of the graph $G$} is defined as the distance $|vw|$ between the ending vertices $v$ and $w$ of this edge; also, the \emph{length $|G|$ of the graph $G$} is defined as the sum of all its edges lengths.

Let $M$ be a finite metric space. We define the number $\mst(M)$ as the length of the shortest tree of the form $(M,E)$. This value is called the \emph{length of minimum spanning tree on $M$}; a tree $G=(M,E)$ such that $|G|=\mst(M)$ is called a \emph{minimum spanning tree on $M$}. Notice that for any $M$ there exists a minimum spanning tree on it. The set of all minimum spanning trees on $M$ is denoted by $\MST(M)$.

\subsection{$\mst$-spectrum of a finite metric space}
Notice that minimum spanning tree may be defined not uniquely. For $G\in\MST(M)$  by $\s(G,M)$ we denote the vector, whose components are the lengths of edges of the tree $G$, ordered descendingly. If it is clear which metric is used, then we write $\s(G)$ instead of $\s(G,M)$. The next result is well-known.

\begin{prop}\label{prop:mst-spect}
For any $G_1,G_2\in\MST(M)$ we have $\s(G_1)=\s(G_2)$.
\end{prop}

Proposition~\ref{prop:mst-spect} explains correctness of the following definition.

\begin{dfn}
For any finite metric space $M$, by $\s(M)$ we denote the vector $\s(G,M)$ for arbitrary $G\in\MST(M)$, and we call this vector by the \emph{$\mst$-spectrum of the space $M$}.
\end{dfn}

\begin{constr}
For any set $M$ by $\cD_k(M)$ we denote  the family of all possible partitions of the $M$ into $k$ its nonempty subsets. Suppose now that $M$ is a metric space and $D=\{M_1,\ldots,M_k\}\in\cD_k(M)$. Put
$$
\a(D)=\min\bigl\{|M_iM_j|:i\ne j\bigr\}.
$$
\end{constr}

The next result is proved in~\cite{TuzMST-GH}.

\begin{prop}\label{prop:spect-calc}
Let $M$ be a finite metric space and $\s(M)=(\s_1,\ldots,\s_{n-1})$. Then
$$
\s_k=\max\bigl\{\a(D):D\in\cD_{k+1}(M)\bigr\}.
$$
\end{prop}

\subsection{$\mst$-spectrum of an arbitrary metric space}\label{subsec:mst-spector}
Now we generalize the concept of $\mst$-spectrum by means of Proposition~\ref{prop:spect-calc}.

\begin{dfn}
For any metric space $X$ and $k\in\N$ we put $\s_k=\sup\{\a(D):D\in\cD_{k+1}(X)\}$ if $\cD_{k+1}(X)\ne\0$, and $\s_k=0$ otherwise. We call the set $\s(X)=\{\s_1,\s_2,\ldots\}$ by \emph{$\mst$-spectrum of $X$}.
\end{dfn}

\begin{rk}
If $\#X=n$, then $\s_k=0$ for $k\ge n$.
\end{rk}

In~\cite{IT_MST}, for metric spaces which can be connected by a finite length tree (see~\cite{INT_MST} for definitions), a necessary condition of minimum spanning trees existence is obtained. It follows from~\cite[Theorem~1]{IT_MST} that for a metric space $X$ which can be connected by a minimum spanning tree $G$ of a finite length, all the edges of $G$ are exact in the following sense. For each edge $e$ and the corresponding vertex sets $X_1$ and $X_2$ of the trees forming the forest $G\sm e$, the length of $e$ equals to the distance between $X_1$ and $X_2$.

Moreover, for $G$ described above and any $\dl>0$ there are finitely many edges of $G$, whose lengths are more than or equal to $\dl$. This enables us to order the edges in such a way that their lengths decrease monotonically. Let $\{e_1,e_2,\ldots\}$ be such an order, and put $r_i=|e_i|$.

\begin{lem}\label{lem:spec_inf}
For any positive integer $k$ consider the partition $D=\{X_1,\ldots,X_{k+1}\}$ of $X$ into vertex sets of the trees forming the forest $G\sm\{e_1,\ldots,e_k\}$. Then $\a(D)=|e_k|$.
\end{lem}

\begin{proof}
Each $p_i\in X_i$ and $p_j\in X_j$, $i\ne j$, are connected by a path in $G$, and the path contains at least one of the edges $e_p$, $p\le k$. Since $G$ is minimal, then $|p_ip_j|\ge|e_p|\ge|e_k|$, thus $|X_iX_j|\ge|e_k|$ and, therefore, $\a(D)\ge|e_k|$. On the other hand, if we choose $X_i$ and $X_j$ in such a way that $e_k$ connects them, then $\a(D)\le|X_iX_j|=|e_k|$, because the edge $e_k$ is exact.
\end{proof}

\begin{lem}\label{lem:spec_inf_2}
Let $D'=\{X'_1,\ldots,X'_{k+1}\}$ be an arbitrary partition of $X$, then $\a(D)\le\a(D')$.
\end{lem}

\begin{proof}
Denote by $E'$ the set of all edges of $G$ connecting different elements of partition $D'$. The set $E'$ contains at least $k$ edges (otherwise, the graph $G$ is disconnected), hence $\min_{e\in E'}|e|\le|e_k|$. Let $e'\in E$ satisfy $|e'|=\min_{e\in E'}|e|$, and let $X'_i$ and $X'_j$ be those elements of $D'$ which are connected by $e'$. Then $\a(D')\le|X'_iX'_j|\le|e'|\le|e_k|=\a(D)$.
\end{proof}

Lemmas~\ref{lem:spec_inf} and~\ref{lem:spec_inf_2} imply the following result.

\begin{cor}\label{cor:spec_inf}
If there exists a minimum spanning tree connecting a metric space $X$, and $\{e_1,e_2,\ldots\}$ are the edges of this tree ordered descendingly, then $\s_k=|e_k|$.
\end{cor}

\section{Maximum Spanning Trees}
\markright{\thesection.~Maximum spanning trees}

Let $M$ be a finite metric space. \emph{Maximum spanning tree $G$ on $M$} is a longest tree among all trees of the form $(M,E)$. By $\xst(M)$ we denote the length of a maximum spanning tree on $M$ , and by $\XST(M)$ we denote  the set of all maximum spanning trees on $M$.

The next construction is useful for description of relations between minimum and maximum spanning trees.

Let $X$ be an arbitrary not pointwise bounded metric space. Choose any $d\ge2\diam X$ and define on $X$ a new distance function: $\r(x,y)=d-|xy|$ for any $x\ne y$, and $\r(x,x)=0$ for any $x$.

\begin{lem}
The function $\r$ is a metric on $X$.
\end{lem}

\begin{proof}
Indeed, it is obvious that $\r$ is positively definite and symmetric. To verify the triangle inequalities, choose any pairwise distinct points $x,y,z\in X$, then
$$
\r(x,y)+\r(y,z)-\r(x,z)=d-|xy|+d-|yz|-d+|xz|\ge d-2\diam X\ge0.
$$
If two of these points coincide, say, if $y=z\ne x$, then $|xy|=|xz|$ and $\r(z,y)=0$, hence
$$
\bigl|\r(x,z)-\r(z,y)\bigr|=\bigl|d-|xz|-0\bigr|=d-|xz|=d-|xy|=\r(x,y)=d-|xz|+0=\r(x,z)+\r(z,y).
$$
\end{proof}

The set $X$ with the metric $\r$ defined above is denoted by $d-X$.

Let $M$ be a finite metric space, $\#M=n$, and $N=n(n-1)/2$. Denote by $\r(M)=(\r_1,\ldots,\r_N)$ the vector constructed from nonzero distances in $M$, ordered descendingly.

\begin{rk}
If $M$ is a finite metric space, and $\r(M)=(\r_1,\ldots,\r_N)$, then $\r(d-M)=(d-\r_N,\ldots,d-\r_1)$.
\end{rk}

\subsection{Duality}
The next Proposition describes a duality between minimum and maximum spanning trees.

\begin{prop}\label{prop:min-max-trees}
Let $M$ be a finite metric space, and $n=\#M$. Then a tree $G=(M,E)$ is a minimum spanning tree on $M$, iff it is a maximum spanning tree on $d-M$.
\end{prop}

\begin{proof}
Let $\r$ stand for the distance function on $d-M$. Then $\r(G)=\sum_{e\in E}\r(e)=d(n-1)-|G|$.
\end{proof}

For $G\in\XST(M)$, by $\S(G,M)$ we denote the vector constructed from the lengths of edges of the tree $G$, ordered ascendingly. If it is clear which metric is in consideration, then we write $\S(G)$ instead of  $\S(G,M)$.

Proposition~\ref{prop:min-max-trees} implies the next result.

\begin{cor}\label{cor:spectrum-duality}
Let $M$ be a finite metric space, $\#M=n$, $d\ge2\diam M$, and $G\in\MST(M)$. Let us denote also by $d$ the vector of the length $n-1$, all whose components equal $d$. Then $G\in\XST(d-M)$ and $\S(G,d-M)+\s(G,M)=d$.
\end{cor}

Corollary~\ref{cor:spectrum-duality} implies an analogue of Proposition~\ref{prop:mst-spect}.

\begin{prop}\label{prop:mst-spect-dual}
For any $G_1,G_2\in\XST(M)$ we have $\S(G_1)=\S(G_2)$.
\end{prop}

Proposition~\ref{prop:mst-spect-dual} motivates the next definition.

\begin{dfn}
For any finite metric space $M$ by $\S(M)$ we denote  the value $\S(G,M)$ for an arbitrary $G\in\XST(M)$, and we call this value by \emph{$\xst$-spectrum of the space $M$}.
\end{dfn}

Let $X$ be an arbitrary metric space, and $A,B\ss X$ be its nonempty subsets. Put
$$
|AB|'=\sup\bigl\{|ab|:a\in A,\,b\in B\bigr\}.
$$
If the metric on $X$ is denoted by $\r$ as well, we put $|AB|'=\r(A,B)'$.

\begin{constr}
For a set $M$, let $\cC_k(M)$ stand for  the family of all coverings of the set $M$ consisting of $k$ nonempty subsets. Now, let $M$ be a metric space, and $C=\{M_1,\ldots,M_k\}\in\cC_k(M)$. Put $\b(C,M)=\max\bigl\{|M_iM_j|':i\ne j\bigr\}$. It it is clear which metric is in consideration, we write $\b(C)$ instead of $\b(C,M)$.
\end{constr}

\begin{lem}\label{lem:dual-dists-sets}
Let $A,B\ss X$ be nonempty subsets of a bounded metric space, and $\r$ be the metric of the space $d-X$. Then $|AB|'=d-\r(A,B)$.
\end{lem}

\begin{proof}
Indeed,
\begin{multline*}
|AB|'=\sup\bigl\{|ab|:a\in A,\,b\in B\bigr\}=\sup\bigl\{d-\r(a,b):a\in A,\,b\in B\bigr\}=\\
=d-\inf\bigl\{\r(a,b):a\in A,\,b\in B\bigr\}=d-\r(A,B).
\end{multline*}
\end{proof}

\begin{lem}\label{lem:beta-alpha-dual}
For any $D=\{M_1,\ldots,M_k\}\in\cD_k(M)$ we have $\b(D,M)=d-\a(D,d-M)$.
\end{lem}

\begin{proof}
Indeed, let $\r$ be the metric of the space $d-M$. Then, by Lemma~\ref{lem:dual-dists-sets}, it holds
\begin{multline*}
\b(D,M)=\max\bigl\{|M_iM_j|':i\ne j\bigr\}=\max\bigl\{d-\r(M_i,M_j):i\ne j\bigr\}=\\
=d-\min\bigl\{\r(M_i,M_j):i\ne j\bigr\}=d-\a(D,d-M).
\end{multline*}
\end{proof}

Up to the end of this section, $M$ stands for a finite metric space consisting of $n$ points, and $\S(M)=(\S_1,\ldots,\S_{n-1})$.

The next result is an analogue of Proposition~\ref{prop:spect-calc}.

\begin{prop}\label{prop:spect-calc-dual}
We have $\S_k=\min\bigl\{\b(D,M):D\in\cD_{k+1}(M)\bigr\}$.
\end{prop}

\begin{proof}
Choose an arbitrary $d\ge2\diam M$, then $\s(d-M)=d-\S(M)$. Put $\s(d-M)=(\s_1,\ldots,\s_{n-1})$ and $\S(M)=(\S_1,\ldots,\S_{n-1})$. By Proposition~\ref{prop:spect-calc} and Lemma~\ref{lem:beta-alpha-dual}, we have
\begin{multline*}
\S_k=d-\s_k=d-\max\bigl\{\a(D,d-M):D\in\cD_{k+1}(d-M)\bigr\}=\\ =\min\bigl\{d-\a(D,d-M):D\in\cD_{k+1}(d-M)\bigr\}=\min\bigl\{\b(D,M):D\in\cD_{k+1}(d-M)\bigr\}.
\end{multline*}
\end{proof}

\begin{rk}
Under the above notations, it holds $\diam M=\S_{n-1}$.
\end{rk}

In the next two Propositions, we use $G=(M,E)\in\XST(M)$, and the edges $e_i\in E$ are supposed to be ordered in such a way that $|e_i|=\S_i$. Moreover, for the convenience reason, we put $|e_0|=\S_0=0$.

\begin{prop}\label{prop:diams-of-xms-partition}
Let $\{M_1,\ldots,M_{k+1}\}\in\cD_{k+1}(M)$ be a partition into the vertex sets of the trees forming the forest $F=G\sm\{e_{n-k},\ldots,e_{n-1}\}$. Then for each $i$ we have $\diam M_i\le|e_{n-k-1}|=\S_{n-k-1}$.
\end{prop}

\begin{proof}
Choose arbitrary $x,y\in M_i$. If $xy\in E$, then $|xy|\le|e_{n-k-1}|$ by the order we have chosen on $E$. If $xy\not\in E$, then consider the unique path $\g$ in $G$ connecting $x$ and $y$. Since $M_i$ is the set of vertices of a tree from the forest $F$, then for each edge $e_j$ of this path it holds $j\le n-k-1$, and, therefore, $|e_j|\le|e_{n-k-1}|$. Since the tree $G$ is maximal, then $|xy|\le|e_j|$ for some $j\le n-k-1$, hence $\diam M_i\le|e_{n-k-1}|$.
\end{proof}

\begin{prop}\label{prop:dual-dists-of-xms-partition}
Let $\{M_1,\ldots,M_{k+1}\}\in\cD_{k+1}(M)$ be a partition into the vertex sets of the trees forming the forest $F=G\sm\{e_1,\ldots,e_k\}$. Then for every $i\ne j$ we have $|M_iM_j|'\le|e_k|=\S_k$.
\end{prop}

\begin{proof}
Indeed, consider arbitrary $M_i, M_j$, $i\ne j$, and let $P_i\in M_i$, $P_j\in M_j$ be arbitrary points. Consider the unique path $\g$ in the tree $G$ connecting $P_i$ and $P_j$. This path contains at least one of the edges thrown out. Let it be $e_p$. If $P_iP_j$ is not an edge in $G$, then $G\cup P_iP_j$ contains a unique cycle $\g\cup P_iP_j$, and the maximality of $G$ implies that each edges of the path $\g$ is not shorter than $P_iP_j$. In particular, $|e_k|\ge|e_p|\ge |P_iP_j|$. If $P_iP_j$ is an edge of $G$, then it coincides with an edge we threw out, say, with $e_p$, and again we have $|e_k|\ge|e_p|=|P_iP_j|$. Thus, $|M_iM_j|'=\max|P_iP_j|\le|e_k|$.
\end{proof}

\subsection{$\xst$-spectrum of an arbitrary metric space}
Similarly with the section~\ref{subsec:mst-spector}, let us generalize the concept of $\xst$-spectrum by means of Proposition~\ref{prop:spect-calc-dual}.

\begin{dfn}
For any metric space $X$ and any $k\in\N$ we put $\S_k=\inf\{\b(D):D\in\cD_{k+1}(X)\}$ if $\cD_{k+1}(X)\ne\0$, and $\S_k=\infty$ otherwise. The set $\S(X)=\{\S_1,\S_2,\ldots\}$ we call the {\em $\xst$-spectrum of $X$}.
\end{dfn}

\begin{rk}
If $\#X=n$, then $\S_k=\infty$ for $k\ge n$.
\end{rk}

\section{Calculation of Distances between Compact Metric Space and Finite Simplexes}
\markright{\thesection.~Calculation of distances between compact metric space and finite simplexes}

We call a metric space $X$ a \emph{simplex}, if all its nonzero distances are the same. Notice that a simplex $X$ is compact, iff it is finite. A simplex consisting of $n$ vertices on the distance $\l$ from each other is denoted by $\l\D_n$. For $\l=1$ the space $\l\D_n$ is denoted by $\D_n$ for short.

\subsection{Distances from a finite metric space to simplexes with greater numbers of points}

\begin{thm}\label{thm:dist-n-simplex-bigger-dim}
Let $M$ be a finite metric space, $n=\#M$. Then for every $m\in\N$, $m>n$, and $\l>0$ we have
$$
2d_{GH}(\l\D_m,M)=\max\{\l,\diam M-\l\}.
$$
\end{thm}

\begin{proof}
Choose an arbitrary $R\in\cR(\l\D_m,M)$. Since $m>n$, then there exists $x\in M$ such that $\#R^{-1}(x)\ge2$, hence $\dis R\ge\l$ and, therefore, $2d_{GH}(\l\D_m,M)\ge\l$.

Put $M=\{x_1,\ldots,x_n\}$ and let $R$ be the correspondence $\bigl\{(i,x_i)\bigr\}_{i=1}^{n-1}\cup\{n,\ldots,m\}\x\{x_n\}$. Then
$$
\dis R=\max\biggl[\l,\,\max_{i\ne j}\Bigl\{\bigl||x_ix_j|-\l\bigr|\Bigr\}\biggr]=\max\{\l,\diam M-\l\},
$$
where the second equality follows from Proposition~\ref{prop:many_abs_dif_and_t}. This implies that $2d_{GH}(\l\D_m,M)\le\max\{\l,\diam M-\l\}$.

If $\diam M\le 2\l$, then $\max\{\l,\diam M-\l\}=\l$, hence $2d_{GH}(\l\D_m,M)=\l$, q.e.d.

If $\diam M>2\l$, then $\max\{\l,\diam M-\l\}=\diam M-\l$. Choose a pair $x,y\in M$ such that $\diam M=|xy|$. Take an arbitrary $R\in\cR(\l\D_m,M)$. Then one of the following conditions holds:
\begin{enumerate}
\item there exists $i\in\l\D_m$ such that $(i,x),(i,y)\in R$, but then $\dis R\ge\diam M>\diam M-\l$;
\item there exist $i\ne j$ such that $(i,x),(j,y)\in R$, and then $\dis R\ge\diam M-\l$.
\end{enumerate}
Thus, for any $R\in\cR(\l\D_m,M)$ we have $\dis R\ge\diam M-\l$, therefore, in the case under consideration the equality $2d_{GH}(\l\D_m,M)=\diam M-\l$ is valid.
\end{proof}

\subsection{Distances from a compact metric space to simplexes with no greater number of points}

\begin{thm}\label{thm:m_less_than_n_corresp}
Let $X$ be a compact metric space. Then for every $m\in\N$, $m\le\#X$, and $\l>0$ there exists an $R\in\cR_{\opt}^0(\l\D_m,X)$ such that the family $\bigl\{R(i)\bigr\}$ is a partition of $X$. In particular, if $m=\#X$, then one can take a bijection as an optimal correspondence  $R$.
\end{thm}

\begin{proof}
Let $R\in\cR(\D_m,X)$ be an arbitrary irreducible correspondence. Since $R$ is irreducible, the condition $R(j)\cap R(k)\ne\0$ for some $j\ne k$ implies $R(j)=R(k)=\{x\}$ for some $x\in X$. In particular, if $\#R(i)>1$ for some $i$, then $R(i)$ does not intersect any $R(p)$, $p\ne i$. Let us introduce the following notation: for $\s\in\cP(\D_m,X)$ we put $\np(\s)$ to be equal to the number of pairs $\{j,k\}$ such that $j\ne k$ and $\s(j)\cap\s(k)\ne\0$. Clearly that for $\s\in\cR(\D_m,X)$ the condition $\np(\s)=0$ is equivalent to that the family $\bigl\{\s(i)\bigr\}$ forms a partition $X$.

Suppose that the family $\bigl\{R(i)\bigr\}$ is not a partition. We show that in this case one can always reconstruct the correspondence $R$ in such a way that the resulting correspondence $\tR$ becomes an irreducible one with $\dis\tR\le\dis R$ and $\np(\tR)<\np(R)$. If $\np(\tR)>0$, then we put $R=\tR$, and repeat this procedure. After a finite number of steps we will get a correspondence $\tR$ such that $\np(\tR)=0$ and $\dis\tR\le\dis R$, q.e.d.

So, let for some $j\ne k$ it holds $R(j)=R(k)=\{x\}$. Since $m\le\#X$, then there exists $i$ such that $\#R(i)>1$, hence $i\not\in\{j,k\}$ and $R(i)$ does not intersect any other $R(p)$. Choose an arbitrary point $x_i\in R(i)$, and construct a new correspondence $\tR$ which coincides with $R$ at all elements of the simplex,  except $i$, $j$, $k$, and
$$
\tR(i)=R(i)\sm\{x_i\}, \quad \tR(j)=\{x_i\}, \quad \tR(k)=\{x\}.
$$
Clearly that $\tR\in\cR(\D_m,X)$ since $R$ is uniquely defined on all elements of the simplex, and each element of $X$ goes to at least one element of the simplex. Besides that, the correspondence $\tR$ is still irreducible, which can be verified directly.

Further, to estimate $\np(\tR)$ let us notice that among all $R(p)$ only $R(i)$ and $R(j)$ are changed, thus, it is sufficient to investigate how the intersections with these two sets changes. Since $R(i)$ does not intersect the remaining $R(p)$, and since $\tR(i)\sqcup\tR(j)=R(i)$, then $\tR(i)$ and $\tR(j)$ does not intersect the remaining $\tR(p)$. Besides that, $\tR(i)\cap\tR(j)=\0$. Thus, the number of the remaining $R(p)$ intersecting with $R(i)$ is the same as the number of the remaining $\tR(p)$ intersecting with $\tR(i)$ (and it is equal $0$). Concerning $R(j)$, the intersection $R(j)\cap R(k)$ is nonempty (and, perhaps, there are some other nonempty intersections with $R(j)$). However, $\tR(j)$ does not intersect any of the remaining $\tR(p)$, hence $\np(\tR)<\np(R)$.

Let us prove that $\dis\tR\le\dis R$.

Put
$$
M(R,p,q)=\max\Bigl\{\big| |x_px_q|-1\big|:x_p\in R(p),\,x_q\in R(q),\,p\ne q\Bigr\}.
$$
Recall that
$$
\dis R=\max\big\{\dl_R=1,\,\max_p\diam R(p),\,\max_{p\ne q}M(R,p,q)\big\},
$$
and
$$
\dis\tR=\max\big\{\dl_\tR,\,\max_p\diam\tR(p),\,\max_{p\ne q}M(\tR,p,q)\big\}.
$$
Clearly that $\dl_{\tR}\le 1\le\dis R$ and $\diam\tR(p)\le\diam R(p)\le\dis R$ for all $p$. To complete the proof it remains to show that for any $p$ and $q$ the inequality $M(\tR,p,q)\le\dis R$ holds.

If $p$ and $q$ are not contained in $\{i,j,k\}$, then $M(\tR,p,q)=M(R,p,q)\le\dis R$.

Now, suppose that one of the indices $p$ and $q$, say $p$, is not contained in $\{i,j,k\}$, but the remaining one is contained. In this case,
\begin{description}
\item[] $M(\tR,p,k)=M(R,p,k)\le\dis R$, because $\tR(p)=R(p)$ and $\tR(k)=R(k)$;
\item[] $M(\tR,p,i)\le M(R,p,i)\le\dis R$, because $\tR(i)\ss R(i)$;
\item[] $M(\tR,p,j)\le M(R,p,i)\le\dis R$, because  $\tR(j)\ss R(i)$.
\end{description}

Finally, consider the case $\{p,q\}\ss\{i,j,k\}$. We have
\begin{description}
\item[] $M(\tR,i,k)\le M(R,i,k)\le\dis R$, because $\tR(i)\ss R(i)$ and $\tR(k)=R(k)$;
\item[] $M(\tR,j,k)\le M(R,i,k)\le\dis R$, because $\tR(j)\ss R(i)$ and $\tR(k)=R(k)$;
\end{description}
and
\begin{multline*}
M(\tR,i,j)=\max\Bigl\{\big| |x_i'x_i|-1\big|:x_i'\in \tR(i)=R(i)\sm\{x_i\}\Bigr\}\le\\
\le\max\Bigl\{1,\big| |x_i'x_i|-1\big|:x_i'\in \tR(i)\Bigr\}\le\max\Bigl\{1,\max\big\{ |x_i'x_i| :x_i'\in \tR(i)\big\}\Bigr\}\le\\
\le\max\big\{1,\diam R(i)\big\}\le\dis R,
\end{multline*}
where the second inequality holds according to Proposition~\ref{prop:max_abs}.

Thus, all the values from the expression for $\dis \tR$ do not exceed $\dis R$, hence $\dis\tR\le\dis R$, q.e.d.
\end{proof}

Theorem~\ref{thm:m_less_than_n_corresp} helps to get a useful formula for Gromov--Hausdorff distance between a compact metric space $X$ and a finite simplex such that the number of points in the simplex does not exceed the cardinality of $X$.

\begin{constr}
For an \textbf{arbitrary\/} metric space $X$, $m\le\#X$, and $D=\{X_1,\ldots,X_m\}\in\cD_m(X)$ we put $R_D=\sqcup\,\bigl(\{i\}\x X_i\bigr)$. Notice that for any $D'\in\cD_m(X)$ which differs from $D$ by renumbering of the elements of the partition $D$, we have $\dis R_D=\dis R_{D'}$.
\end{constr}

\begin{notation}
For $D=\{X_1,\ldots,X_m\}\in\cD_m(X)$ let us put
$$
\diam D=\max\{\diam X_1,\ldots,\diam X_m\}.
$$
\end{notation}

\begin{prop}\label{prop:disRD}
Let $X$ be an arbitrary metric space, and $m\in\N$, $m\le\#X$. Then for any $\l>0$ and $D\in\cD_m(X)$ it holds
$$
\dis R_D=\max\{\diam D,\,\l-\a(D),\,\b(D)-\l\}.
$$
\end{prop}

\begin{proof}
Let $D=\{X_1,\ldots,X_m\}$. By definition of distortion,
$$
\dis R_D=\sup\{\diam D,\,\bigl|\l-|xy|\bigr|:x\in X_p,\,y\in X_q,\,1\le p<q\le m\}.
$$
By Proposition~\ref{prop:many_abs_dif}, we have
$$
\sup\{\bigl|\l-|xy|\bigr|:x\in X_p,\,y\in X_q,\,1\le p<q\le m\}=\max\bigl\{\l-\a(D),\,\b(D)-\l\bigr\}.
$$
\end{proof}

\begin{prop}\label{prop:GH-dist-RD}
Let $X$ be a compact metric space. Then for every $m\in\N$, $m\le\#X$, and $\l>0$ it holds
$$
2d_{GH}(\l\D_m,X)=\inf_{D\in\cD_m(X)}\dis R_D.
$$
\end{prop}

\begin{proof}
By Theorem~\ref{thm:m_less_than_n_corresp}, there exists an $R\in\cR_{\opt}^0(\l\D_m,X)$ such that the family $\bigl\{R(i)\bigr\}$ is a partition of $X$. Thus, $d_{GH}(\l\D_m,X)$ is achieved at some $R_D$, $D\in\cD_m(X)$.
\end{proof}

\begin{cor}\label{cor:GH-dist-alpha-beta}
Let $X$ be a compact metric space, and $m\in\N$, $m\le\#X$. Then for any $\l>0$ we have
$$
2d_{GH}(\l\D_m,X)=\inf\Bigl\{\max\bigl(\diam D,\,\l-\a(D),\,\b(D)-\l\bigr):D\in\cD_m(X)\Bigr\}.
$$
\end{cor}

\subsection{Distance from finite metric space to simplexes with the same number of points}

For any metric space $X$ we put
$$
\e(X)=\inf\bigl\{|xy|:x,y\in X,\,x\ne y\bigr\}.
$$
Notice that $\e(X)\le\diam X$, and the equality holds, iff $X$ is a simplex. Besides that, if $M$ is a finite metric space consisting of $n$ points, and $\s(M)=(\s_1,\ldots,\s_{n-1})$, then  $\e(M)=\s_{n-1}$.

\begin{thm}\label{thm:dist-n-simplex-same-dim}
Let $M$ be a finite metric space, $\#M=n$, $\s(M)=(\s_1,\ldots,\s_{n-1})$, $\S(M)=(\S_1,\ldots,\S_{n-1})$, $\l>0$. Then
$$
2d_{GH}(\l\D_n,M)=\max\{\l-\s_{n-1},\,\S_{n-1}-\l\}=\max\bigl\{\l-\e(M),\,\diam M-\l\bigr\}.
$$
More exactly, if $\s_{n-1}+\S_{n-1}\le2\l$, then $2d_{GH}(\l\D_n,M)=\l-\s_{n-1}$, otherwise $2d_{GH}(\l\D_n,M)=\S_{n-1}-\l$.
\end{thm}

\begin{proof}
For each $D\in\cD_n(X)$ it holds $\diam D=0$. Besides that, for all such $D$ we have $\a(D)=\s_{n-1}$ and $\b(D)=\S_{n-1}$. It remains to apply Corollary~\ref{cor:GH-dist-alpha-beta}.
\end{proof}

\subsection{Distance from a compact metric space to simplexes having at most the same number of points}

In~\cite{TuzMST-GH} the following result is proved.

\begin{prop}[\cite{TuzMST-GH}]\label{mst-spector-GH}
Let $X$ be a finite metric space, $\l\ge2\diam X$, and $\s(X)=(\s_1,\ldots,\s_{n-1})$. Then $2d_{GH}(\l\D_{k+1},X)=\l-\s_k$ for all $k=1,\ldots,n-1$.
\end{prop}

The next theorem generalizes Proposition~\ref{mst-spector-GH} to the case of compact metric spaces and also weaken the restrictions on the parameter $\l$.

\begin{thm}\label{thm:spec}
Let $X$ be a compact metric space, $\s(X)=\{\s_1,\s_2,\ldots\}$ be the $\mst$-spectrum of $X$, and $\l\ge\diam X+\s_k$. Then $2d_{GH}(\l\D_{k+1},X)=\l-\s_k$ for all $k\in\N$, $k+1\le\#X$.
\end{thm}

\begin{proof}
Put $m=k+1$ and choose an arbitrary partition $D\in\cD_m(X)$. Due to Proposition~\ref{prop:disRD}, we have
$$
\dis R_D=\max\{\diam D,\,\l-\a(D),\,\b(D)-\l\}.
$$
Notice that $\diam D\le\diam X$ and $\b(D)-\l\le\diam X-\l<\diam X$. Further, since $\l\ge\diam X+\s_k$, then $\l-\a(D)\ge\l-\s_k\ge\diam X$, and hence $\dis R_D=\l-\a(D)$. So, due to Proposition~\ref{prop:GH-dist-RD}, we have $2d_{GH}(\l\D_m,X)=\l-\s_k$.
\end{proof}

\begin{notation}
Let $X$ be an arbitrary metric space. Put $d_m(X)=\inf\bigl\{\diam D:D\in\cD_m(X)\bigr\}$, if $\cD_m(X)\ne\0$, and $d_m(X)=\infty$ otherwise.
\end{notation}

\begin{rk}
If $X$ is a finite metric space and $n=\#X$, then $d_n(X)=0$.
\end{rk}

Recall that a \emph{clique\/} in a simple graph is any its subgraph which is a complete graph. A graph is said to be  an \emph{$m$-clique}, if it contains a spanning subgraph which is a disjoint union of $m$ cliques. Notice that each graph having  $n$ vertices is $n$-clique.

Let $X$ be an arbitrary metric space, and $\dl\ge0$. By $G_\dl(X)$ we denote the graph with the vertex set $X$, where $v,w\in X$, $v\ne w$, are connected by edge iff $|vw|\le\dl$ (this graph is infinite, generally speaking). Define the number $\dl_m(X)$ to be equal to the infimum of $\dl\ge0$ such that $G_\dl(X)$ is $m$-clique. If there is no such $\dl$, then we put $\dl_m(X)=\infty$.

\begin{prop}
For an arbitrary metric space $X$ the equality $d_m(X)=\dl_m(X)$ holds.
\end{prop}

\begin{proof}
Indeed, $\cD_m(X)=\0$ iff $\#X<m$, and the latter is equivalent to nonexistence of an $m$-clique graph with the vertex set $X$. So, in this case the required equality holds.

Let $\#X\ge m$. The condition $d_m(X)=\infty$ is equivalent to the fact that for any partition $D=\{X_i\}\in\cD(X)$ we have $\diam X_i=\infty$ for some $i$. The latter condition is equivalent to nonexistence of an $m$-clique subgraph in the complete graph with the vertex set $X$, such that all its edges do not exceed some value $\dl$. Thus, in this case the required equality holds as well.

Now, let $d_m(X)<\infty$. Consider the family of partitions $D_i\in\cD_m(X)$ such that $d_i=\diam D_i\to d_m(X)$. Then the graph $G_{d_i}$ contains an $m$-clique subgraph such that each $D_i$ lies in some its clique. The latter implies that $\dl_m(X)\le d_m(X)$.

Conversely, let $\dl_i$ be a decreasing sequence such that $\dl_i\to\dl_m(X)$. By $G_i$ we denote some $m$-clique subgraph of $G_{\dl_i}$. Let $H_i$ be the subgraph in $G_i$ that is equal to the disjoint union of $m$ cliques. Denote by $X_i^p$ the vertex sets of the cliques of the graph $H_i$, then $\diam X_i^p\le\dl_i$. Put $D_i=\{X_i^p\}$, and we get $D_i\in\cD_m(X)$ and $\diam D_i\le\dl_i$, so $d_m(X)\le\dl_m(X)$.
\end{proof}

\begin{thm}\label{thm:small1}
Let $X$ be a compact metric space, $\s(X)=\{\s_1,\s_2,\ldots\}$ be the $\mst$-spectrum of $X$, and $\l<\diam X+\s_k$. Assume that for some $k\in\N$, $k+1\le\#X$, the equality $d_{k+1}(X)=\diam X$ is valid. Then $2d_{GH}(\l\D_{k+1},X)=\diam X$.
\end{thm}

\begin{proof}
Put $m=k+1$. Recall that, in accordance to Proposition~\ref{prop:disRD}, for any $D\in\cD_m(X)$ the relation
$$
\dis R_D=\max\{\diam D,\,\l-\a(D),\,\b(D)-\l\}
$$
is valid. Besides, $\diam D\le\diam X$ and $\b(D)-\l\le\diam X-\l<\diam X$. Since $\l<\diam X+\s_k$, then there exists  $D\in\cD_m(X)$ such that $\l-\diam X<\a(D)$, therefore, for such $D$ we have $\l-\a(D)<\diam X$. Gathering all those inequalities, we conclude that for such $D$ the inequality $\dis R_D\le\diam X$ holds. Thus, due to Proposition~\ref{prop:GH-dist-RD}, we have $2d_{GH}(\l\D_m,X)\le\diam X$.

On the other hand, since $d_m(X)=\diam X$, then for any $R_D\in\cR^0_{\opt}(\l\D_m,X)$ we have
$$
\diam X\ge2d_{GH}(\l\D_m,X)=\dis R_D\ge\diam D\ge d_m(X)=\diam X,
$$
so $2d_{GH}(\l\D_m,X)=\diam X$.
\end{proof}

\begin{thm}\label{thm:small2}
Let $X$ be a compact metric space. Then for any $m\in\N$, $m\le\#X$, and any $0<\l\le(\diam X)/2$ the equality $2d_{GH}(\l\D_m,X)=\max\{d_m(X),\diam X-\l\}$ holds.
\end{thm}

\begin{proof}
Choose an arbitrary partition $D=\{X_i\}\in\cD_m(X)$. Due to proposition~\ref{prop:disRD}, we have
$$
\dis R_D=\max\{\diam D,\,\l-\a(D),\,\b(D)-\l\}.
$$
Notice that $\diam D\le\diam X$, $\l-\a(D)\le\l<\diam X$, and $\b(D)-\l\le\diam X-\l<\diam X$, therefore, $\dis R_D\le\diam X$, and if $\diam D<\diam X$, then $\dis R_D<\diam X$. In particular, $2d_{GH}(\l\D_m,X)<\diam X$ in this case.

Let $d_m(X)=\diam X$, then
$$
\diam X\ge\dis R_D\ge\diam D\ge d_m(X)=\diam X,
$$
hence, $\dis R_D=\diam X$ for all $D\in\cD_m$, and, due to Proposition~\ref{prop:GH-dist-RD}, we have
$$
2d_{GH}(\l\D_m,X)=\dis R_D=\diam X=\max\bigl\{d_m(X),\diam X-\l\bigr\}.
$$

Now, let  $d_m(X)<\diam X$. Then, in accordance to the definition of infimum, there exists $D'\in\cD_m(X)$ such that  $\diam D'<\diam X$. As we have already mentioned above, the latter implies that $\dis R_{D'}<\diam X$, and hence, $2d_{GH}(\l\D_m,X)<\diam X$. On the other hand, due to Theorem~\ref{thm:m_less_than_n_corresp}, there exists a $D\in\cD_m(X)$ such that $R_D\in\cR^0_{\opt}(\l\D_m,X)$. But then $\dis R_D\le\dis R_{D'}<\diam X$ and $\diam D<\diam X$.

Also, notice that the inequality $\diam D<\diam X$ implies the equality $\b(D)=\diam X$. Besides, since
$$
\b(D)+\a(D)=\diam X +\a(D)\ge\diam X
$$
in this case, then the assumption $\l\le(\diam X)/2$ implies that  $\b(D)+\a(D)\ge 2\l$, and hence, $\b(D)-\l\ge\l-\a(D)$ and
$$
\dis R_D=\max\{\diam D,\,\b(D)-\l\}.
$$

Thus, $D\in\cD_m(X)$, $R_D\in\cR^0_{\opt}(\l\D_m,X)$, and $\diam D<\diam X$. So, if $\diam D\le\diam X-\l$, then $\dis R_D=\max\{\diam D,\,\diam X-\l\}=\diam X-\l=\max\bigl\{d_m(X),\,\diam X-\l\bigr\}$.

If $\diam D\ge\diam X-\l$, then
$$
\dis R_D=\max\{\diam D,\,\diam X-\l\}=\diam D\ge\max\{d_m(X),\,\diam X-\l\}.
$$
Let us prove the inverse inequality. To do that, consider a sequence $D_i\in\cD_m(X)$ such that $\diam X>\diam D_i\to d_m(X)$. For each $i$ we have $\b(D_i)=\diam X$ and $\l-\a(D_i)\le\b(D_i)-\l$, and hence,
$$
\dis R_{D_i}=\max\{\diam D_i,\diam X-\l\}\to\max\bigl\{d_m(X),\diam X-\l\bigr\}.
$$
But $R_D\in\cR^0_{\opt}(\l\D_m,X)$, therefore, $\dis R_D\le \dis R_{D_i}$ for any $i$. Passing to the limit, we conclude that $\dis R_D\le \max\bigl\{d_m(X),\,\diam X-\l\bigr\}$, q.e.d.
\end{proof}

Now we apply our technique to calculate the distances between a finite $n$-point metric space and the simplex $t\D_{n-1}$.

\begin{thm}\label{thm:minus_one}
Let  $X$ be a finite metric space, $\#X=n\ge2$, $\l>0$. Then
$$
2d_{GH}(\l\D_{n-1},X)=\max\{\s_{n-1}, \l-\s_{n-2},\S_{n-1}-\l\}.
$$
Moreover, the correspondence  that takes one of the simplex's elements to a pair of the closest points from $X$, and that is one-to-one on the remaining elements, is optimal.
\end{thm}

\begin{proof}
Apply Corollary~\ref{cor:GH-dist-alpha-beta} again and conclude that the doubled distance sought for is equal to the minimum of distortions of the correspondences $R_D$ generated by partitions $D$ of the space $X$ into $n-1$ subsets. Since  $\#X=n$, then all such partitions have the following form: $n-2$ elements of a partition are single points, and one element of the partition consists of two points. Let the two-point element be equal to $\{x_i,x_j\}$. Then
$$
\dis R_D=\max\big\{|x_ix_j|,\,\l-\a(D),\,\b(D)-\l\big\},
$$
$\a(D)=\min\big\{|x_px_q|:p\ne q,\,\{p,q\}\ne\{i,j\}\big\}$, and $\b(D)=\max\big\{|x_px_q|:p\ne q,\,\{p,q\}\ne\{i,j\}\big\}$, i.e., the minimum and the maximum are taken over all the distances from $X$, except the single one. Notice that if we write down  all nonzero distances in a finite metric space in the increasing order, then the resulting sequence has the form $\{\s_{n-1},\s_{n-2},\ldots,\S_{n-2},\S_{n-1}\}$. Therefore,
$$
\dis R_D=\begin{cases}
\max\{\s_{n-1},\,\l-\s_{n-2},\,\S_{n-1}-\l\}, &\text{if $|x_ix_j|=\s_{n-1}$},\\
\max\{\S_{n-1},\,\l-\s_{n-1},\,\S_{n-2}-\l\}, &\text{if $|x_ix_j|=\S_{n-1}$},\\
\max\big\{|x_ix_j|,\,\l-\s_{n-1},\,\S_{n-1}-\l\big\}, &\text{in the remaining cases}.
\end{cases}
$$
Since $\S_{n-1}$ is the maximal distance in $X$, and $\s_{n-1}$ is the minimal one, then
$$
\max\{\S_{n-1},\,\l-\s_{n-1},\,\S_{n-2}-\l\}=\max\{\S_{n-1},\,\l-\s_{n-1}\}\ge \max\{\s_{n-1},\, \l-\s_{n-2},\,\S_{n-1}-\l\}.
$$
Further, $\max\big\{|x_ix_j|,\,\l-\s_{n-1},\,\S_{n-1}-\l\big\}\ge\max\{\s_{n-1},\,\l-\s_{n-2},\,\S_{n-1}-\l\big\}$ because $\s_{n-1}\le |x_ix_j|$ for all $i$ and $j$. So,
$$
2d_{GH}(\l\D_{n-1},X)=\min_D\dis R_D=\max\{\s_{n-1},\,\l-\s_{n-2},\,\S_{n-1}-\l\}.
$$
Moreover, the right hand side of the latter equality is the distortion of the correspondence which takes one of the simplex's elements to a pair of the closest elements from the space $X$, and that is one-to-one on the remaining parts of the spaces.
\end{proof}

\subsection{Examples}
However, in general case calculation of the Gromov--Hausdorff distance to a simplex remains a difficult problem. Here we discuss some examples of the Gromov--Hausdorff distance calculation between the simplexes  $t\D_2$ and a four-point metric space $X=\{x_1,x_2,x_3,x_4\}$ with a distance matrix
$$
\left(\begin{array}{cccc}
0 & a & b & d \\
a & 0 & c & e \\
b & c & 0 & f \\
d & e & f & 0
\end{array}\right).
$$

In accordance to Theorem~\ref{thm:m_less_than_n_corresp} and Corollary~\ref{cor:GH-dist-alpha-beta}, to calculate the distance it suffices to consider only irreducible correspondences $R_D\in\cR(t\D_2,X)$ generated by partitions $D=\{X_1,X_2\}$ of the set $X$ into nonempty subsets. Evidently, there are seven such partitions, namely, four into a single element subset and a three-element subset, and three into a pair of two-element subsets. If such a partition is fixed, then in accordance to Proposition~\ref{prop:disRD} its distortion has the following form:
$$
\dis R_D=\max\big\{\diam D,\,t-\a(D),\,\b(D)-t\big\}.
$$
If $X_1=\{x_1\}$, $X_2=\{x_2,x_3,x_4\}$, then
$$
\dis R_D=\max\big\{\max\{c,e,f\},\,t-\min\{a,b,d\},\,\max\{a,b,d\}-t\big\},
$$
and if  $X_1=\{x_1,x_2\}$, $X_2=\{x_3,x_4\}$, then
$$
\dis R_D=\max\big\{\max\{a,f\},\,t-\min\{b,c,d,e\},\,\max\{b,c,d,e\}-t\big\}.
$$
To obtain an answer in more concrete form, we need to make some assumptions concerning the distances in $X$.

\subsubsection{Distance function is cumbersome, but can be expressed in terms of the spectra}
Assume that the distances are ordered as follows: $a<e<b<c<f<d$.  In this case the spectra are $\s=\{b,e,a\}$ and $\S=\{c,f,d\}$.  In accordance to the above reasoning, the distance is calculated as the following minimum of seven maxima:
\begin{multline}\label{mult:1}
2d_{GH}(t\D_2,X)=\min\Big\{
\max\big\{f,t-a,d-t\big\},
\max\big\{d,t-a,c-t\big\},
\max\big\{d,t-b,f-t\big\}, \\
\max\big\{c,t-e,d-t\big\},
\max\big\{f,t-e,d-t\big\},
\max\big\{b,t-a,d-t\big\},
\max\big\{d,t-a,f-t\big\}
\Big\}.
\end{multline}
Consider the third, fourth, and sixth maxima in the right hand part of Formula~(\ref{mult:1}), and notice that the remaining maxima are greater than or equal to one of these three. This can be verified directly. For example, $\max\big\{d,t-a,c-t\big\}\ge\max\big\{b,t-a,d-t\big\}$, because $d\ge b$ and $d\ge d-t$. Thus,
\begin{equation}\label{mult:1.1}
2d_{GH}(t\D_2,X)=\min\Big\{
\max\big\{d,t-b,f-t\big\}, \\
\max\big\{c,t-e,d-t\big\},
\max\big\{b,t-a,d-t\big\}
\Big\}.
\end{equation}

\begin{lem}
For $t\in[0,a+c]$ we have $2d_{GH}(t\D_2,X)=\max\big\{b,t-a,d-t\big\}$.
\end{lem}

\begin{proof}
It suffices to verify that in the chosen segment the first two maxima in Formula~(\ref{mult:1.1}) are not less than the third one.

To compare the first and the third maxima, notice that: $b<d$ for all $t$; since $t\le a+c$, then $t-a\le c<d$; and, at last, $d-t\le d$.

To compare the second and the third maxima, notice that: $b<c$ for all $t$; since $t\le a+c$, then $t-a\le c$.
\end{proof}

\begin{lem}
For $t\in[a+c,e+d]$ we have $2d_{GH}(t\D_2,X)=\max\big\{c,t-e,d-t\big\}$.
\end{lem}

\begin{proof}
It suffices to verify that in the chosen segment the first and the third maxima in Formula~(\ref{mult:1.1}) are not less than the second one.

To compare the first and the second maxima, notice that: $c<d$ for all $t$; since $t\le e+d$, then $t-e<d$; and, at last, $d-t\le d$.

To compare the third and the second maxima, notice that: $t-a>t-e$ for all $t$; since $t\ge a+c$, then $t-a\ge c$.
\end{proof}

\begin{lem}
For $t\in[e+d,\infty]$ we have $2d_{GH}(t\D_2,X)=\max\big\{d,t-b,f-t\big\}$.
\end{lem}

\begin{proof}
It suffices to verify that in the chosen segment the second and the third maxima in Formula~(\ref{mult:1.1}) are not less than the first one.

To compare the second and the first maxima, notice that: $t-e>t-b$, and also $d-t>f-t$ for all $t$; since $t\ge e+d$, then $t-e\ge d$.

To compare the third and the first maxima, notice that: $t-a>t-b$, and also $d-t>f-t$ for all $t$; since $t\ge e+d$, then $t-a>t-e\ge d$.
\end{proof}

To visualize the obtained result let us fix the following values of the distances:  $a=3$, $b=4$, $c=5$, $d=6.5$, $e=3.5$, $f=6$ (the triangle inequalities can be verified directly). The graph of the function $f(t)=2d_{GH}(t\D_2,X)$ is depicted in Figure~\ref{fig:exam_1}.

\ig{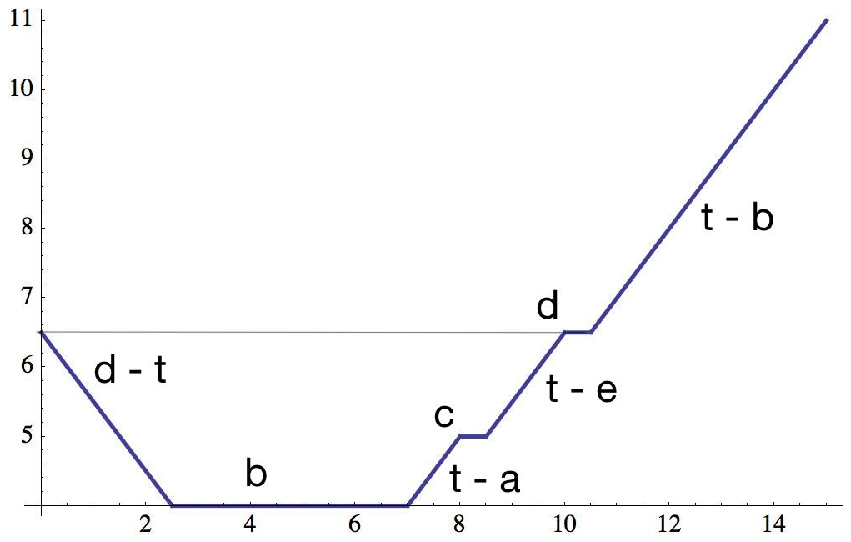}{0.33}{fig:exam_1}{Graph of the function $f(t)=2d_{GH}(t\D_2,X)$ for $a=3$, $e=3.5$, $b=4$, $c=5$,  $f=6$, $d=6.5$.}

Notice that in the case under consideration all the six distances belong to one of the two spectra. Therefore the function $f(t)$ can be expressed in terms of the spectra (in this case all the elements of the spectrum $\s$, and also $\S_1$ and $\S_3$, are used). But this is not always true.

\subsubsection{Distance function is easier, but can not be expressed in terms of spectra}
Assume now that the distances are ordered as follows: $a<b<c<d<e<f$.  In this case the spectra are $\s=\{d,b,a\}$ and $\S=\{d,e,f\}$.  The value $c$ does not belong to $\s$, because the corresponding edge forms a cycle together with the two edges of the least lengths. Now, the Gromov--Hausdorff distance between $X$ and the simplex $t\D_2$ has the form
\begin{multline}\label{mult:2}
2d_{GH}(t\D_2,X)=\min\Big\{
\max\big\{f,t-a,d-t\big\},
\max\big\{f,t-a,e-t\big\},
\max\big\{e,t-b,f-t\big\}, \\
\max\big\{c,t-d,f-t\big\},
\max\big\{f,t-b,e-t\big\},
\max\big\{e,t-a,f-t\big\},
\max\big\{d,t-a,f-t\big\}
\Big\}.
\end{multline}

Notice that all the maxima in the right hand part of Formula~(\ref{mult:2}) are not less than the forth maximum $\max\big\{c,t-d,f-t\big\}$, and hence
$$
2d_{GH}(t\D_2,X)=\max\big\{c,t-d,f-t\big\}.
$$

To visualize the result, we include the graph of the function $g(t)=2d_{GH}(t\D_2,X)$ for some specific values of the distances, see Figure~\ref{fig:exam_2}.

\ig{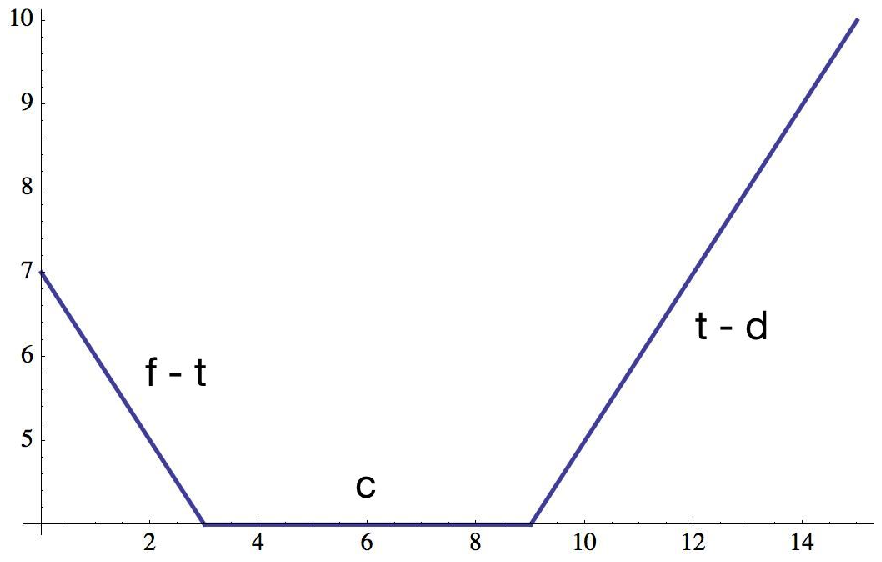}{0.33}{fig:exam_2}{Graph of the function $g(t)=2d_{GH}(t\D_2,X)$ for $a=2$, $b=3$, $c=4$, $d=5$,  $e=6$, $f=7$.}

In spite of the fact that the function $g$ is simpler than $f$ from the first example, it can not be expressed in terms of the spectra, because $c$ does not belong to them. Notice also that in this case the value $c$ is equal to  $\dl_2(X)$ (i.e., to the $2$-clique number of $X$).

\subsubsection{Distances to the simplexes do not distinct non-isometric spaces}
It is well-known that the set of pairwise distances does not completely define the corresponding metric space. One of simple examples can be obtain as follows. On the set $\{x_1,x_2,x_3,x_4\}$ we consider the following two distance matrices that differ in the transposition of the distances $|x_1x_4|$ and $|x_3x_4|$:
$$
S_1=\left(\begin{array}{cccc}
0 & a & b & d \\
a & 0 & c & e \\
b & c & 0 & f \\
d & e & f & 0
\end{array}\right), \qquad
S_2=\left(\begin{array}{cccc}
0 & a & b & f \\
a & 0 & c & e \\
b & c & 0 & d \\
f & e & d & 0
\end{array}\right).
$$
Notice that if all the pairwise distances are close to each other, then the triangle inequalities are valid for the both matrices.

Assume that $a<b<c<d<f<e$, and let $X_i$ be the space with the distance matrix $S_i$, $i=1,\;2$. The sets of the pairwise distances of these spaces are the same, but the spaces are not isometric. For example, the unique minimum spanning tree in $X_1$ is the star--tree centered at the point $x_1$, and the unique minimum spanning tree in $X_2$ is the path $x_2x_1x_3x_4$. But the spectra $\s$, the maximum spanning trees, and the spectra $\S$ of these spaces are the same.

\begin{ass}\label{ass:examp_non_isom}
Under the above assumptions, the Gromov--Hausdorff distances from the spaces $X_i$ to any simplex are the same.
\end{ass}

\begin{proof}
Due to Theorem~\ref{thm:dist-n-simplex-bigger-dim}, the distances from $X_i$ to the simplexes consisting of five and more vertices are the same, because $\diam X_1=\diam X_2=e$. The distances to the simplexes $t\D_4$ are the same in accordance to Theorem~\ref{thm:dist-n-simplex-same-dim}, because the largest and the smallest distances in the spaces $X_i$ are the same. Further, the distances to the simplexes $t\D_3$ are the same in accordance to Theorem~\ref{thm:minus_one}, because the spaces $X_i$ have equal spectra $\s$ and $\S$. The distances to the single-point simplex are the same, due to Proposition~\ref{prop:GH_simple}, Item~(\ref{prop:GH_simple:1}), because the diameters of the spaces $X_i$ are the same.  It remains to calculate the distances to two-point simplexes. It can be done similarly to the above examples.

For the space $X_1$ the distance can be calculated as follows:
\begin{multline}\label{mult:3}
2d_{GH}(t\D_2,X_1)=\min\Big\{\max\big\{e,t-a,d-t\big\},
\max\big\{f,t-a,e-t\big\},
\max\big\{e,t-b,f-t\big\}, \\
\max\big\{c,t-d,e-t\big\},
\max\big\{f,t-b,e-t\big\},
\max\big\{e,t-a,f-t\big\},
\max\big\{d,t-a,e-t\big\}\Big\};
\end{multline}
and for the space $X_2$ the distance has the form
\begin{multline}\label{mult:4}
2d_{GH}(t\D_2,X_2)=\min\Big\{\max\big\{e,t-a,f-t\big\},
\max\big\{f,t-a,e-t\big\},
\max\big\{e,t-b,d-t\big\}, \\
\max\big\{c,t-d,e-t\big\},
\max\big\{d,t-b,e-t\big\},
\max\big\{e,t-a,f-t\big\},
\max\big\{f,t-a,e-t\big\}\Big\}.
\end{multline}

Notice that, due to the chosen order, all the maxima in the right hand parts of Formulas~(\ref{mult:3}) and~(\ref{mult:4}) are not less than the fourth maximum $\max\big\{c,t-d,e-t\big\}$, so
$$
2d_{GH}(t\D_2,X_1)=\max\big\{c,t-d,e-t\big\}=2d_{GH}(t\D_2,X_2).
$$
\end{proof}

\begin{cor}\label{cor:non_isom}
In the space of four-point metric spaces there exists an open subset $U$ such that for any $X\in U$ one can find a four-point metric space $X'$ that is not isometric to $X$, and such that $d_{GH}(X,t\D_m)=d_{GH}(X',t\D_m)$ for all $t>0$ and $m\in\N$.
\end{cor}

\begin{rk}
Notice that the distances to simplexes from the spaces $X_i$ do not depend on $f$, and so, we have found in fact an infinite (continuum) family of pairwise non-isometric finite metric spaces such that the distances to all simplexes are the same. A natural problem is to describe all such families of metric spaces. In particular, to find out if it is possible to construct similar examples for infinite metric spaces.
\end{rk}

\end{document}